\documentclass[11pt,reqno]{amsart}
\usepackage{hyperref}
\usepackage{url}
\usepackage{amssymb}
\usepackage{tikz,enumerate}
\usepackage{diagbox}
\usepackage{appendix}
\allowdisplaybreaks[4]

\newtheorem{theorem}{Theorem}

\newtheorem{conj}{Conjecture}
\newtheorem{lemma}{Lemma}

\theoremstyle{definition}

\theoremstyle{remark}
\newtheorem{rem}{Remark}
\numberwithin{equation}{section}
\numberwithin{theorem}{section}
\numberwithin{lemma}{section}
\numberwithin{defn}{section}
\numberwithin{corollary}{section}
\DeclareMathOperator{\spt}{spt}

\begin{document}
\title[The smallest parts function associated with $\omega(q)$]
 {The smallest parts function associated with $\omega(q)$}

\author{Liuquan Wang and Yifan Yang}
\address{School of Mathematics and Statistics, Wuhan University, Wuhan 430072, Hubei, People's Republic of China}
\address{Mathematics Division, National Center for Theoretical Sciences, Taipei 10617, Taiwan}
\email{wanglq@whu.edu.cn;mathlqwang@163.com}

\address{Department of Mathematics, National Taiwan University, Taipei 10617, Taiwan}
\email{yangyifan@ntu.edu.tw}

\subjclass[2010]{Primary 11F33, 11P83; Secondary 05A17,	11F11, 11F37}

\keywords{partitions; mock theta function; smallest parts functions; congruences modulo powers of 5; Eisenstein series}

\begin{abstract}
We establish two families of congruences modulo powers of 5 for the Fourier coefficients of $(2E_2(2\tau)-E_2(\tau))\eta(2\tau)^{-1}$, where $E_2(\tau)$ is the weight 2 Eisenstein series and $\eta(\tau)$ is the Dedekind eta function. This allows us to prove similar congruences for two smallest parts functions. The first function is $\spt_{\omega}(n)$, which was introduced by Andrews, Dixit and Yee and associated with Ramanujan/Watson's third order mock theta function $\omega(q)$. The second one is $\spt_{C5}(n)$, which appeared in the work of Garvan and Jennings-Shaffer. Moreover, we confirm two conjectural congruences of Wang.
\end{abstract}

\maketitle
\section{Introduction}
A partition of a positive integer $n$ is a way of writing it as a sum of positive integers in non-increasing order. The number of partitions of $n$ is usually denoted as $p(n)$ and we agree that $p(0)=1$. The generating function of $p(n)$ can be written as
\begin{align}\label{p(n)-gen}
\sum_{n=0}^\infty p(n)q^n=\frac{1}{(q;q)_\infty}.
\end{align}
Here and throughout this paper, we assume that $|q|<1$ and adopt the $q$ series notation:
\begin{align}
&(a;q)_\infty=\prod\limits_{n=0}^\infty (1-aq^n),  \\
&(a_1,a_2, \dots, a_m;q)_\infty=\prod\limits_{k=1}^{m}(a_k;q)_\infty,
\end{align}
and for $n\in \mathbb{Z}$,
\begin{align}
(a;q)_n=\frac{(a;q)_\infty}{(aq^n;q)_\infty}.
\end{align}
Ramanujan discovered that $p(n)$ satisfies some nice congruences modulo powers of 5, 7 and 11. One of them is for $k\geq 1$ and $n\geq 0$,
\begin{align}\label{p(n)-5-power}
p(5^kn+\delta_k)\equiv 0 \pmod{5^k},
\end{align}
where $\delta_k$ is the least nonnegative integer such that $24\delta_k\equiv 1$ (mod $5^k$).

Around 2008, Andrews \cite{Andrews-spt} introduced the smallest parts function $\spt (n)$, which counts the total number of appearances of the smallest parts in all partitions of $n$. He showed that the generating function of $\spt (n)$ satisfies
\begin{align}
\sum_{n=0}^\infty \spt (n)q^n=\frac{1}{(q;q)_\infty} \left( \sum_{n=1}^\infty \frac{nq^n}{1-q^n}+\sum_{n=1}^\infty \frac{(-1)^nq^{n(3n+1)/2}(1+q^n)}{(1-q^n)^2} \right). \label{spt-gen}
\end{align}
Andrews also proved some congruences modulo 5, 7 and 13 for $\spt(n)$ such as
\begin{align}
\spt (5n+4) &\equiv 0 \pmod{5}. \label{spt-mod5}
\end{align}
After that, this function and some analogous functions have been investigated in many works. See the recent survey of Chen \cite{Chen} for more related works.

In 2012, employing some classical modular equations, Garvan \cite{Garvan-TAMS} successfully  derived congruences for $\spt(n)$ modulo powers of 5, 7 and 13.  For instance, he proved that for $k\geq 3$,
\begin{align}
\spt (5^k n +\delta_k)+ 5\spt (5^{k-2}n+\delta_{k-2}) \equiv 0 \pmod{5^{2k-3}}. \label{spt-5-power-original}
\end{align}
This together with \eqref{spt-mod5} implies that for $k\geq 1$,
\begin{align}
\spt (5^kn+\delta_k) \equiv 0 \pmod{5^{\lfloor \frac{k+1}{2}\rfloor}}. \label{spt-5-power}
\end{align}

In 2015, Andrews, Dixit and Yee  \cite{ADY} defined some new partition functions associated with Ramanujan/Watson mock theta functions. In particular, for the third order mock theta function
\begin{align}
\omega(q):=&\sum\limits_{n=0}^{\infty}\frac{q^{2n^2+2n}}{(q;q^2)_{n+1}^{2}},
\end{align}
they gave a combinatorial interpretation to its coefficients. Let $p_{\omega}(n)$ be the number of partitions of $n$ in which all odd parts are less than twice the smallest part. It follows from definition that
\begin{align}\label{p-omega-gen}
\sum\limits_{n=1}^{\infty}p_{\omega}(n)q^n=\sum\limits_{n=1}^{\infty}\frac{q^n}{(1-q^n)(q^{n+1};q)_{n}(q^{2n+2};q^2)_{\infty}}.
\end{align}
Andrews et al.\ \cite{ADY} showed that
\begin{equation}\label{p-omega-2nd}
\sum\limits_{n=1}^{\infty}p_{\omega}(n)q^n=q\omega(q).
\end{equation}
They then considered the smallest parts function associated to $p_{\omega}(n)$. Let $\spt_{\omega}(n)$ count the number of smallest parts in the partitions enumerated by $p_{\omega}(n)$. The generating function for $\spt_{\omega}(n)$ is
\begin{align}
\sum\limits_{n=1}^{\infty}\spt_{\omega}(n)q^n=\sum\limits_{n=1}^{\infty}\frac{q^n}{(1-q^n)^2(q^{n+1};q)_{n}(q^{2n+2};q^2)_{\infty}}. \label{spt-omega-gen}
\end{align}
An alternative representation was also given \cite[Lemma 6.1]{ADY}:
\begin{align}
\sum\limits_{n=1}^{\infty}\spt_{\omega}(n)q^n=\frac{1}{(q^2;q^2)_{\infty}}\sum\limits_{n=1}^{\infty}\frac{nq^n}{1-q^n}
+\frac{1}{(q^2;q^2)_{\infty}}\sum\limits_{n=1}^{\infty}\frac{(-1)^n(1+q^{2n})q^{n(3n+1)}}{(1-q^{2n})^2}.  \label{spt-omega-new-gen}
\end{align}
Based on this formula, they proved several interesting congruences such as
\begin{align}
\spt_{\omega}(5n+3) &\equiv 0 \pmod{5}, \label{spt-omega-mod5-a}\\
\spt_{\omega}(10n+7) &\equiv 0 \pmod{5}, \label{spt-omega-mod5-b}\\
\spt_{\omega}(10n+9) &\equiv 0 \pmod{5}. \label{spt-omega-mod5-c}
\end{align}

It is worthy mention that the generating function \eqref{spt-omega-gen} also appears in the work of Garvan and Jennings-Shaffer
 \cite[p.\ 612]{Garvan-Shaffer}. In that paper, Garvan and Jennings-Shaffer introduced some $\spt$-type functions corresponding to the Bailey pairs $A1$, $A3$, $A5$, $A7$, $C1$, $C5$, $E2$, and $E4$ in Slater's work \cite{Slater}.
In particular, they defined $\spt_{C1}(n)$ by \eqref{spt-omega-gen}, which means
\begin{align}\label{add-spt-omega-C1}
\spt_{\omega}(n)=\spt_{C1}(n).
\end{align}
A combinatorial interpretation to $\spt_{C1}(n)$ was also given in \cite{Garvan-Shaffer}, which agrees with the combinatorial meaning of $\spt_{\omega}(n)$.

Furthermore, Garvan and Jennings-Shaffer \cite{Garvan-Shaffer} also defined the $\spt$-type function $\spt_{C5}(n)$ by
\begin{align}
\sum_{n=1}^\infty \spt_{C5}(n)q^n=\sum_{n=1}^\infty \frac{q^{(n^2+n)/2}}{(1-q^n)^2(q^{n+1};q)_n(q^{2n+2};q^2)_\infty}. \label{spt-C5-gen}
\end{align}
Interestingly,  this function is closely related to $\spt(n)$ and $\spt_{C1}(n)$ in the following way \cite[Corollary 2.10]{Garvan-Shaffer}:
\begin{align}
\spt(n/2)=\spt_{C1}(n)-\spt_{C5}(n), \label{relation}
\end{align}
where $\spt(n/2)$ is zero if $n$ is odd. By establishing some equalities between some crank-type functions associated with $\spt_{C1}(n)$ and $\spt_{C5}(n)$, it was proved in \cite{Garvan-Shaffer} that
\begin{align}
\spt_{C1}(5n+3) &\equiv 0 \pmod{5}, \label{C1-mod5} \\
\spt_{C5}(5n+3) &\equiv 0 \pmod{5}. \label{C5-mod5}
\end{align}
Congruence \eqref{C1-mod5} is the same as \eqref{spt-omega-mod5-a} but the proofs given in \cite{ADY} and \cite{Garvan-Shaffer} are different.

Recently,  among other results, Wang \cite{Wang-JNT} discovered some new congruences for $p_\omega(n)$. Moreover, he pointed out that as a natural consequence of \eqref{spt-omega-new-gen}, we have
\begin{align}
\sum\limits_{n=0}^{\infty}\spt_{\omega}(2n+1)q^n=\frac{(q^2;q^2)_{\infty}^8}{(q;q)_{\infty}^5}.
\end{align}
Wang then proposed the following conjecture.
\begin{conj}\label{Wang-conj}
(Cf.\ \cite[Conjecture 4.1]{Wang-JNT}) For any integers $k\ge 1$ and $n\ge 0$, we have
\begin{align}
\spt_{\omega}\Big(2\cdot 5^{2k-1}n+\frac{7\cdot 5^{2k-1}+1}{12} \Big) &\equiv 0 \pmod{5^{2k-1}},\label{mod-5-odd}\\
\spt_{\omega}\Big(2 \cdot 5^{2k}n+\frac{11\cdot 5^{2k}+1}{12} \Big)  &\equiv 0 \pmod{5^{2k}}. \label{mod-5-even}
\end{align}
\end{conj}

The main goal of this paper is to confirm the above conjecture. Toward this we note that the relations \eqref{add-spt-omega-C1} and \eqref{relation} imply
\begin{align}
\spt_{\omega}(2n+1)=\spt_{C5}(2n+1). \label{omega-C5-odd}
\end{align}
Therefore, it suffices to establish suitable congruences modulo powers of 5 for $\spt_{C5}(n)$.

Recall the weight 2 Eisenstein series
\begin{align}
E_2(\tau):=1-24\sum_{n=1}^\infty \frac{nq^n}{1-q^n}, \quad q=e^{2\pi i \tau}, \quad \mathrm{Im}\tau>0.
\end{align}
Let the sequence $c(n)$ be defined by
\begin{align}
\sum_{n=0}^\infty c(n)q^n:=\frac{2E_2(2\tau)-E_2(\tau)}{(q^2;q^2)_\infty}.
\end{align}
We find that (see \eqref{C5-relation-proved})
\begin{align}
\spt_{C5}(n)=\frac{1}{24}\left(c(n)-p(n/2)+12np(n/2) \right), \label{C5-relation}
\end{align}
where as before, we agree that $p(n/2)=0$ if $n$ is an odd number. Since congruences modulo powers of 5 for $p(n)$ are already known, we can get 5-adic properties of $\spt_{C5}(n)$ from those of $c(n)$.

Using the theory of modular forms, we establish the following congruences for $c(n)$.
\begin{theorem}\label{thm-c(n)}
For $k\geq 1$ and $n\geq 0$, we have
\begin{align}
c\left(5^{2k-1}n+\frac{7\cdot 5^{2k-1}+1}{12} \right) &\equiv 0 \pmod{5^{2k-1}}, \label{c(n)-cong-1} \\
c\left(5^{2k}n+\frac{11\cdot 5^{2k}+1}{12} \right)  &\equiv 0 \pmod{5^{2k}}. \label{c(n)-cong-2}
\end{align}
\end{theorem}
With this theorem, we get congruences modulo powers of 5 for $\spt_{C5}(n)$.
\begin{theorem}\label{thm-C5}
For $k\geq 1$ and $n\geq 0$, we have
\begin{align}
\spt_{C5}\left( 5^{2k-1}n+\frac{7\cdot 5^{2k-1}+1}{12} \right) & \equiv 0 \pmod{5^{2k-1}}, \label{C5-cong-power-1} \\
\spt_{C5}\left( 5^{2k}n+\frac{11 \cdot 5^{2k}+1}{12} \right) & \equiv 0 \pmod{5^{2k}}. \label{C5-cong-power-2}
\end{align}
\end{theorem}
As a consequence, from \eqref{spt-5-power}, \eqref{omega-C5-odd}, \eqref{C5-cong-power-1} and \eqref{C5-cong-power-2} we get
\begin{theorem}\label{thm-conj}
Conjecture \ref{Wang-conj} is true. Moreover, for $k\geq 1$ and $n\geq 0$ we have
\begin{align}\label{spt-omega-even}
\spt_{\omega}(2(5^kn+\delta_k))\equiv 0 \pmod{5^{\left\lfloor\frac{k+1}{2}\right\rfloor}}.
\end{align}
\end{theorem}

The paper is organized as follows. In Section \ref{sec-pre} we establish a modular equation and discuss the action of Atkin's $U$-operator on certain functions.  These are needed for finding representations for the generating functions of the sequences in Theorem \ref{thm-c(n)}. In Section \ref{sec-proof} we shall give proofs to Theorems \ref{thm-c(n)}-\ref{thm-conj}. Some auxiliary formulas are presented in the Appendix. 

\section{Preliminaries}\label{sec-pre}
We denote the upper half complex plane by $\mathbb{H}:=\{\tau \in \mathbb{C}|\mathrm{Im} \tau>0\}$.
The Dedekind eta function is defined as
\begin{align}
\eta(\tau):=q^{1/24}(q;q)_\infty.
\end{align}
The key ingredients in our proof will be four functions: $\rho, t, F$ and $Z$, and we now define them.

We define two functions $\rho=\rho(\tau)$ and $t=t(\tau)$ as follows:
\begin{align}
\rho=\frac{\eta^2(2\tau)\eta^4(5\tau)}{\eta^4(\tau)\eta^2(10\tau)}, \quad t=\frac{\eta^2(5\tau)\eta^2(10\tau)}{\eta^2(\tau)\eta^2(2\tau)}.
\end{align}
The function $\rho$ is a hauptmodul of modular functions on $\Gamma_{0}(10)$. The function $t$ is a modular function on $\Gamma_0(10)$. The orders of $\rho$ and $t$ at cusps $c/d$ ($(c,d)=1$) of $\Gamma_0(10)$ are as follows:
\begin{table}[htbp]
\begin{tabular}{|c|c|c|c|c|}
  \hline
  $d$ & 1 & 2 & 5 & 10 \\
  \hline
  $\mathrm{ord}_{c/d}\rho$ & $-1$ & 0 & 1 & 0 \\
  \hline
   $\mathrm{ord}_{c/d}t$ & $-1$ & $-1$ & 1  & 1 \\
  \hline
\end{tabular}
\end{table}

We have the following relation between $\rho$ and $t$.
\begin{lemma}\label{lem-sigma-t}
We have
\begin{align}
\rho^2=\rho +5\rho t-t.
\end{align}
\end{lemma}
\begin{rem}
The function $\rho$ was used by Gordon and Hughes \cite{Gordon-Hughes} in studying partitions into distinct parts. The function $t$  was used by Xiong \cite{Xiong} in the study of cubic partitions.
\end{rem}
We define two more useful functions:
\begin{align}
&Z=Z(\tau):=\frac{\eta(50\tau)}{\eta(2\tau)}, \\
&F=F(\tau):=\frac{1}{24}\left(50E_2(10\tau)-25E_2(5\tau)-2E_2(2\tau)+E_2(\tau) \right).
\end{align}
The function $Z$ is a modular function on $\Gamma_0(50)$, while the function $F$ is a modular form of weight 2 on $\Gamma_0(10)$.

For a function $f:\mathbb{H}\rightarrow \mathbb{C}$ and positive integer $m$, Atkin's $U$-operator $U_m$ is defined by
\begin{align}
U_m(f)(\tau):=\frac{1}{m}\sum_{\lambda=0}^{m-1} f\left(\frac{\tau+\lambda}{m} \right),  \quad \tau \in \mathbb{H}.
\end{align}
From this definition, it is not difficult to see that
\begin{align}
U_m\left(\sum_{n=-\infty}^\infty f_nq^n\right)=\sum_{n=-\infty}^\infty f_{mn}q^n.
\end{align}

We are now going to study the effects of the action $U_5$ to the functions $Ft^i$, $F\rho t^i$, $ FZ t^i$ and $FZ \rho t^i$. For this we need the following lemma.
\begin{lemma}\label{lem-modular}
For $0\leq \lambda \leq 4$ we define
\begin{align}
t_{\lambda}(\tau):=t\left(\frac{\tau+\lambda}{5}\right).
\end{align}
Then in the polynomial ring $\mathbb{C}(t)[X]$ we have
\begin{align}
X^5+\sum_{j=0}^4a_j(t)X^j=\prod\limits_{\lambda=0}^4 (X-t_\lambda), \label{poly}
\end{align}
where
\begin{align*}
&a_0(t)=-t, \quad a_1(t)=-t(2\cdot 5+5^2t), \quad a_2(t)=-t(11\cdot 5+2\cdot 5^3 t+5^4t^2), \\
&a_3(t)=-t(28\cdot 5 +11\cdot 5^3 t+ 2\cdot 5^5 t^2+5^6t^3), \\
&a_4(t)=-t(7\cdot 5^2+ 28\cdot 5^3t+11\cdot 5^5 t^2+2\cdot 5^7t^3+5^8t^4).
\end{align*}
\end{lemma}
\begin{proof}
First we prove that
\begin{align}
a_0(t)=-\prod\limits_{\lambda=0}^4 t_\lambda=-t. \label{a0}
\end{align}
We have
\begin{align*}
\prod\limits_{\lambda=0}^4 t_\lambda&=\left( \prod\limits_{\lambda=0}^4 e^{2\pi i \cdot \frac{\tau+\lambda}{5}} \right)\cdot
\prod\limits_{n=1}^\infty \prod\limits_{\lambda=0}^4 \frac{\left(1-e^{2\pi i n(\tau+\lambda)}\right)^2\left(1-e^{4\pi i n(\tau+\lambda)}\right)^2}{\left(1-e^{2\pi in \cdot \frac{\tau+\lambda}{5}}\right)^2\left(1-e^{2\pi in \cdot \frac{2(\tau+\lambda)}{5}}\right)^2} \nonumber \\
&=q\left(\prod\limits_{n=1}^\infty (1-q^n)^2(1-q^{2n})^2 \right)^5 \cdot \prod\limits_{\begin{smallmatrix} n=1 \\ 5|n\end{smallmatrix}}^\infty
\frac{1}{(1-q^{n/5})^{10}(1-q^{2n/5})^{10}} \\
&\quad \cdot \prod\limits_{\begin{smallmatrix} n=1 \\ 5 \nmid n \end{smallmatrix}}^{\infty} \frac{1}{(1-q^n)^2(1-q^{2n})^2} \\
&=q\frac{(q^5;q^5)_{\infty}^2(q^{10};q^{10})_{\infty}^2}{(q;q)_{\infty}^2(q^2;q^2)_{\infty}^2} \\
&=t.
\end{align*}
Now we use \eqref{a0} to rewrite \eqref{poly} as
\begin{align}
X^5+\sum_{j=0}^4a_j(t)X^j=-t\prod\limits_{\lambda=0}^4 (1-Xt_{\lambda}^{-1}).
\end{align}
Hence we have
\begin{align}
a_j(t)=(-1)^{j+1}ts_j(t_0^{-1}, \dots, t_4^{-1}), \quad 0\leq j \leq 4, \label{a-s-relation}
\end{align}
where the $s_j$ are the elementary symmetric functions. Now it suffices to evaluate $s_j$ explicitly. For this we use the fact that $5U_5(t^{-j})=p_j$ where
\begin{align}
p_j=\sum\limits_{\lambda=0}^4 t_{\lambda}^{-j}, \quad j \in \mathbb{Z}.
\end{align}
By comparing the orders at cusps and direct computations, we find  that
\begin{align*}
&U_5(t^{-1})=-2-5t, \quad U_5(t^{-2})=-2-5^3t^2, \\
&U_5(t^{-3})=46-5^5t^3, \quad U_5(t^{-4})=-42\cdot 5 -5^7t^4.
\end{align*}
The above formulas can also be found in the work of Xiong \cite{Xiong}.
Using Newton's relation, we get
\begin{align*}
&s_1=p_1=-2\cdot 5 -5^2 t, \\
&s_2=(s_1p_1-p_2)/2= 11\cdot 5+ 2\cdot 5^3  t+5^4t^2, \\
&s_3=(s_2p_1-s_1p_2+p_3)/3=-(28\cdot 5 + 11\cdot 5^3 t+ 2\cdot 5^5 t^2+5^6t^3), \\
&s_4=(s_3p_1-s_2p_2+s_1p_3-p_4)/4=7\cdot 5^2+ 28\cdot 5^3t+11\cdot 5^5 t^2+ 2\cdot 5^7t^3+5^8t^4.
\end{align*}
Together with \eqref{a-s-relation}, we complete the proof of the lemma.
\end{proof}
Lemma \ref{lem-modular} implies the following important recurrence relation, which is the key in our proof.
\begin{lemma}\label{lem-key}
For $u:\mathbb{H}\rightarrow \mathbb{C}$ and $j\in \mathbb{Z}$, we have
\begin{align}
U_5(ut^j)=-\sum_{l=0}^4a_l(t)U_5(ut^{j+l-5}).
\end{align}
\end{lemma}
\begin{proof}
By Lemma \ref{lem-modular} we have
\begin{align}
t_\lambda^5+\sum_{l=0}^4 a_l(t)t_{\lambda}^l=0.
\end{align}
Multiplying both sides by $u_\lambda t_{\lambda}^{j-5}$ where $u_{\lambda}(\tau):=u((\tau+\lambda)/5)$ and then summing over $\lambda$ from 0 to 4, we get the desired result.
\end{proof}
\begin{rem}
The deduction of Lemmas \ref{lem-modular} and \ref{lem-key} follow the arguments used in the work of Paule and Radu \cite{Paule-Radu}. However, the functions we use here are different.
\end{rem}

Lemma \ref{lem-key}  reduces the evaluation of $U_5(ut^j)$ to the evaluations of $U_5(ut^j)$ for $-4\leq j \leq 0$. Next we will set $u$ to be $F$, $F\rho$, $FZ$ and $FZ\rho$, respectively, to get their explicit representations. We shall use $\lceil x \rceil$ (resp.\ $\lfloor x\rfloor$) to denote the least (resp.\ greatest) integer no less (resp.\ no greater) than $x$.

\begin{lemma}\label{lem-U-Ft}
For $i\in \mathbb{Z}$ we have
\begin{align}\label{U5-Fti-eq}
U_5(Ft^i)=F\sum_{j=\left\lceil \frac{i}{5}\right\rceil}^\infty m(i,j)t^j,
\end{align}
where the values of $m(i,j)$ for $-4\leq i \leq 0$ are given in Group I of the Appendix, and for other $i$, $m(i,j)$ satisfies
\begin{align}\label{recurrence}
m(i,j)=&~\Big( 7\cdot 5^2 m(i-1,j-1)+ 28\cdot 5^3 m(i-1,j-2) +11 \cdot 5^5  m(i-1,j-3)\nonumber \\
&~ + 2\cdot 5^7m(i-1,j-4)+5^8m(i-1,j-5)  \Big) +\Big(28\cdot 5 m(i-2,j-1) \nonumber \\
&~ + 11\cdot 5^3 m(i-2,j-2)  +2\cdot 5^5 m(i-2,j-3)+5^6m(i-2,j-4) \Big)\nonumber \\
&~ + \Big(11\cdot 5 m(i-3,j-1)+ 2\cdot 5^3 m(i-3,j-2)+5^4m(i-3,j-3)  \Big) \nonumber \\
&~ +\Big(2\cdot 5 m(i-4,j-1)+5^2m(i-4,j-2) \Big)+m(i-5,j-1).
\end{align}
\end{lemma}
\begin{proof}
For $-4\leq i \leq 0$, \eqref{U5-Fti-eq} has been verified in Group I of the Appendix. The theorem follows by Lemma \ref{lem-key} and mathematical induction on $i$.
\end{proof}
\begin{rem}
By definition it is not hard to see that for each $i$, $m(i,j)\neq 0$ only for finitely many $j$. Therefore, the sum in Lemma \ref{lem-U-Ft} is  indeed a finite sum. Here and later in Lemmas \ref{lem-U-Fat}--\ref{lem-U-FZat} as well as Theorem \ref{thm-c(n)-gen}, for simplicity we just allow the summation index $j$ to range to $\infty$ although the sums there are all finite.
\end{rem}

Similarly, we can prove the following results.
\begin{lemma}\label{lem-U-Fat}
For $i\in \mathbb{Z}$ we have
\begin{align}
U_5(F\rho t^i)=F\Big((\sum_{j=\left\lceil \frac{i}{5}\right\rceil}^\infty x_0(i,j)t^j)+\rho(\sum_{j=\left\lceil \frac{i}{5}\right\rceil}^\infty x_1(i,j)t^j) \Big),
\end{align}
where the values of $x_0(i,j)$ and $x_1(i,j)$ for $-4\leq i \leq 0$ are given in Group \uppercase\expandafter{\romannumeral2} of the Appendix, and for other $i$, both $x_0(i,j)$ and $x_1(i,j)$ satisfy the recurrence relation in \eqref{recurrence}.
\end{lemma}

\begin{lemma}\label{lem-U-FZt}
For $i\in \mathbb{Z}$ we have
\begin{align}
U_5(FZ t^i)=F\Big((\sum_{j=\left\lceil \frac{i+1}{5}\right\rceil}^\infty y_0(i,j)t^j)+\rho(\sum_{j=\left\lceil \frac{i+1}{5}\right\rceil}^\infty y_1(i,j)t^j) \Big),
\end{align}
where the values of $y_0(i,j)$ and $y_1(i,j)$ for $-4\leq i \leq 0$ are given in Group \uppercase\expandafter{\romannumeral3} of the Appendix, and for other $i$, both $y_0(i,j)$ and $y_1(i,j)$ satisfy the recurrence relation in \eqref{recurrence}.
\end{lemma}

\begin{lemma}\label{lem-U-FZat}
For $i\in \mathbb{Z}$ we have
\begin{align}
U_5(FZ\rho t^i)=F\Big((\sum_{j=\left\lceil \frac{i+2}{5}\right\rceil}^\infty z_0(i,j)t^j)+\rho(\sum_{j=\left\lceil \frac{i+2}{5}\right\rceil}^\infty z_1(i,j)t^j) \Big),
\end{align}
where the values of $z_0(i,j)$ and $z_1(i,j)$ for $-4\leq i \leq 0$ are given in Group \uppercase\expandafter{\romannumeral4} of the Appendix, and for other $i$, both $z_0(i,j)$ and $z_1(i,j)$ satisfy the recurrence relation in \eqref{recurrence}.
\end{lemma}

Let $\pi(n)$ denote the 5-adic order of $n$ and we agree that $\pi(0)=\infty$.
We now analyze the 5-adic orders of the coefficients appeared in Lemmas \ref{lem-U-Ft}--\ref{lem-U-FZat}. Since they satisfy the common recurrence relation \eqref{recurrence}, we first derive the following result.
\begin{lemma}\label{lem-common}
Let $g(i,j)$ be integers which satisfy the recurrence relation \eqref{recurrence}. Suppose there exists an integer $l$ and a constant $\gamma$ such that for $l\leq i\leq l+4$ we have
\begin{align}
\pi(g(i,j))\geq \left\lfloor \frac{5j-i+\gamma}{3} \right\rfloor. \label{suppose}
\end{align}
Then \eqref{suppose} holds for any $i\in \mathbb{Z}$.
\end{lemma}
\begin{proof}
Suppose \eqref{suppose} holds for $l \leq i \leq k-1$ where $k>l+4$ is an integer. By \eqref{recurrence} we have
\begin{align*}
&\pi(g(k,j))\\
\geq &\min \Big\{\left\lfloor \frac{5(j-1)-(k-1)+\gamma}{3}\right\rfloor+2 , \left\lfloor \frac{5(j-2)-(k-1)+\gamma}{3} \right\rfloor+3,  \\ & \left\lfloor \frac{5(j-3)-(k-1)+\gamma}{3}\right\rfloor +5,
 \left\lfloor \frac{5(j-4)-(k-1)+\gamma}{3} \right\rfloor+7, \\
 & \left\lfloor \frac{5(j-5)-(k-1)+\gamma}{3}\right\rfloor +8, \left\lfloor \frac{5(j-1)-(k-2)+\gamma}{3}\right\rfloor+1, \\
 & \left\lfloor \frac{5(j-2)-(k-2)+\gamma}{3}\right\rfloor +3, \left\lfloor \frac{5(j-3)-(k-2)+\gamma}{3}\right\rfloor +5, \\
 & \left\lfloor \frac{5(j-4)-(k-2)+\gamma}{3}\right\rfloor+6,  \left\lfloor \frac{5(j-1)-(k-3)+\gamma}{3}\right\rfloor+1, \\
 &\left\lfloor \frac{5(j-2)-(k-3)+\gamma}{3}\right\rfloor+3,  \left\lfloor \frac{5(j-3)-(k-3)+\gamma}{3} \right\rfloor +4, \\
 &\left\lfloor \frac{5(j-1)-(k-4)+\gamma}{3} \right\rfloor +1, \left\lfloor \frac{5(j-2)-(k-4)+\gamma}{3}\right\rfloor +2, \\
 &\left\lfloor \frac{5(j-1)-(k-5)+\gamma}{3} \right\rfloor  \Big\} \\
=&\left\lfloor \frac{5j-k+\gamma}{3}\right\rfloor.
\end{align*}
Thus \eqref{suppose} holds when $i=k$. By induction it holds for any $i> l+4$. Similarly, by induction we can show that it holds for any $i<l$.
\end{proof}

\begin{lemma}\label{lem-mxyz-ord}
For $i,j \in \mathbb{Z}$, we have
\begin{align}
&\pi(m(i,j))\geq \left\lfloor \frac{5j-i}{3} \right\rfloor, \label{m-ord} \\
&\pi(x_0(i,j))\geq \left\lfloor \frac{5j-i}{3} \right\rfloor, \label{x0-ord} \\
&\pi(x_1(i,j))\geq \left\lfloor \frac{5j-i+2}{3} \right\rfloor, \label{x1-ord} \\
&\pi(y_0(i,j))\geq \left\lfloor \frac{5j-i-1}{3} \right\rfloor, \label{y0-ord} \\
&\pi(y_1(i,j))\geq \left\lfloor \frac{5j-i+1}{3} \right\rfloor, \label{y1-ord} \\
&\pi(z_0(i,j))\geq \left\lfloor \frac{5j-i-2}{3} \right\rfloor, \label{z0-ord} \\
&\pi(z_1(i,j))\geq \left\lfloor \frac{5j-i+1}{3} \right\rfloor. \label{z1-ord}
\end{align}
\end{lemma}
\begin{proof}
We have verified these inequalities for $-4\leq i \leq 0$ using the data in the Appendix. Therefore, by Lemma \ref{lem-common} we know that they hold for any $i\in \mathbb{Z}$.
\end{proof}
\section{Proofs of the theorems}\label{sec-proof}
In order to prove Theorem \ref{thm-c(n)}, we establish explicit representations for the generating functions of the subsequences of $c(n)$ in the congruences \eqref{c(n)-cong-1} and \eqref{c(n)-cong-2}.

Let
\begin{align}
L_0=L_0(\tau):=2E_2(2\tau)-E_2(\tau).
\end{align}
For $i\geq 1$, we recursively define
\begin{align}
&L_{2i-1}(\tau):=U_5(Z(\tau)L_{2i-2}(\tau)), \\
&L_{2i}(\tau):=U_5(L_{2i-1}(\tau)).
\end{align}
For example,
\begin{align}
L_1=(q^{10};q^{10})_\infty \sum_{n=0}^\infty c(5n+3)q^{n+1}.
\end{align}
By induction it is not difficult to show that for $i\geq 1$,
\begin{align}
L_{2i-1}&=(q^{10};q^{10})_\infty \sum_{n=0}^\infty c\left(5^{2i-1}n+\frac{7\cdot 5^{2i-1}+1}{12} \right)q^{n+1}, \label{L-odd} \\
L_{2i}&=(q^2;q^2)_\infty \sum_{n=0}^\infty c\left(5^{2i}n+\frac{11\cdot 5^{2i}+1}{12} \right)q^{n+1}. \label{L-even}
\end{align}
It is possible to represent $L_i$ using the functions $F$, $\rho$ and $t$. First we define two sequences which will appear as coefficients in such representations.

We define $\{a(i,j)\}_{i,j\geq 1}$ and $\{b(i,j)\}_{i,j\geq 1}$ as \\
(i) $a(1,1)=49\cdot 5$, $a(1,2)=6\cdot 5^4$, $a(1,3)=5^6$, and $a(1,j)=0$ for $j\geq 4$; \\
(ii) $b(1,1)=-5^3$, $b(1,2)=-5^5$, $b(1,j)=0$ for $j\geq 2$; \\
(iii) For $i\geq 1$ and $k\geq 1$,
\begin{align}
a(2i,k)&=\sum_{j=1}^\infty \left(a(2i-1,j)m(j,k)+b(2i-1,j)x_0(j,k) \right), \label{a-even} \\
b(2i,k)&=\sum_{j=1}^\infty  b(2i-1,j)x_1(j,k), \label{b-even} \\
a(2i+1,k)&=\sum_{j=1}^\infty \left(a(2i,j)y_0(j,k)+b(2i,j)z_0(j,k) \right), \label{a-odd} \\
b(2i+1,k)&=\sum_{j=1}^\infty \left(a(2i,j)y_1(j,k)+b(2i,j)z_1(j,k) \right), \label{b-odd}
\end{align}

\begin{theorem}\label{thm-c(n)-gen}
For $i\geq 1$ we have
\begin{align}
L_i=F\left(\sum_{j=1}^\infty a(i,j)t^j+ \rho \sum_{j=1}^\infty b(i,j)t^j \right). \label{Li-rep}
\end{align}
\end{theorem}
\begin{proof}
It is not difficult to find that
\begin{align}
L_1=F\left((49\cdot 5t+6\cdot 5^4 t^2+5^6 t^3)+\rho (-5^3t-5^5t^2)  \right).
\end{align}
Hence \eqref{Li-rep} holds for $i=1$.

Now we suppose \eqref{Li-rep} holds for $2i-1$ where $i\geq 1$. That is,
\begin{align}
L_{2i-1}=F\left(\sum_{j=1}^\infty a(2i-1,j)t^j+ \rho \sum_{j=1}^\infty b(2i-1,j)t^j \right). \label{Li-proof-1}
\end{align}
Applying $U_5$ to both sides of \eqref{Li-proof-1}, by Lemmas \ref{lem-U-Ft} and \ref{lem-U-Fat} we have
\begin{align*}
L_{2i}&=U_5(L_{2i-1})\\
&= \sum_{j=1}^\infty a(2i-1,j)U_5(Ft^j)+\sum_{j=1}^\infty b(2i-1,j)U_5(F\rho t^j) \\
&=F\sum_{j=1}^\infty a(2i-1,j)\sum_{k=1}^\infty m(j,k)t^k \\
&\quad \quad +F\sum_{j=1}^\infty b(2i-1,j)\left(\sum_{k=1}^\infty x_0(j,k)t^k+\rho \sum_{k=1}^\infty x_1(j,k)t^k \right) \\
&=F\sum_{k=1}^\infty \Big(\sum_{j=1}^\infty \big(a(2i-1,j)m(j,k)+b(2i-1,j)x_0(j,k)\big)t^k \\
&\quad \quad +\rho \sum_{j=1}^\infty b(2i-1,j)x_1(j,k)t^k  \Big).
\end{align*}
By \eqref{a-even} and \eqref{b-even}, this implies
\begin{align}
L_{2i}=F\left(\sum_{j=1}^\infty a(2i,j)t^j+ \rho \sum_{j=1}^\infty b(2i,j)t^j \right). \label{Li-proof-2}
\end{align}
Therefore, \eqref{Li-rep} holds for $2i$.

Next, multiplying \eqref{Li-proof-2} by $Z$ and then applying $U_5$ to both sides, by Lemmas \ref{lem-U-FZt} and \ref{lem-U-FZat}  we have
\begin{align*}
L_{2i+1}&=U_5(ZL_{2i})\\
&= \sum_{j=1}^\infty a(2i,j)U_5(FZt^j)+\sum_{j=1}^\infty b(2i,j)U_5(FZ\rho t^j) \\
&=F\sum_{j=1}^\infty a(2i,j)\left(\sum_{k=1}^\infty y_0(j,k)t^k+\rho \sum_{k=1}^\infty y_1(j,k)t^k\right)  \\
&\quad \quad +F\sum_{j=1}^\infty b(2i,j)\left(\sum_{k=1}^\infty z_0(j,k)t^k+\rho \sum_{k=1}^\infty z_1(j,k)t^k \right) \\
&=F\sum_{k=1}^\infty \Big(\sum_{j=1}^\infty \big(a(2i,j)y_0(j,k)+b(2i,j)z_0(j,k)\big)t^k \\
&\quad \quad +\rho \sum_{j=1}^\infty \big(a(2i,j)y_1(j,k)+b(2i,j)z_1(j,k)\big)t^k  \Big).
\end{align*}
By \eqref{a-odd} and \eqref{b-odd}, this implies that \eqref{Li-rep} holds for $2i+1$. By induction we complete the proof.
\end{proof}

In order to prove Theorem \ref{thm-c(n)}, we still need to analyze the 5-adic orders of $a(i,j)$ and $b(i,j)$.
\begin{lemma}\label{lem-ab-order}
For $i,j\geq 1$ we have
\begin{align}
&\pi(a(2i-1,j))\geq 2i-1+\left\lfloor \frac{5j-5}{3}\right\rfloor, \label{a-odd-order} \\
&\pi(a(2i,j))\geq 2i+\left\lfloor \frac{5j-4}{3}\right\rfloor, \label{a-even-order} \\
&\pi(b(2i-1,j))\geq 2i-1+\left\lfloor \frac{5j-3}{3} \right\rfloor, \label{b-odd-order} \\
&\pi(b(2i,j))\geq 2i+1+\left\lfloor \frac{5j-5}{3}\right\rfloor. \label{b-even-order}
\end{align}
\end{lemma}
\begin{proof}
We verify \eqref{a-odd-order} and \eqref{b-odd-order} for $i=1$ directly. Now suppose \eqref{a-odd-order} and \eqref{b-odd-order} hold for $i$.

By \eqref{a-even} and Lemma \ref{lem-mxyz-ord} we have
\begin{align*}
\pi(a(2i,k)) &\geq \min\limits_{j\geq 1} \Big\{\pi(a(2i-1,j))+\pi(m(j,k)), \pi(b(2i-1,j))+\pi(x_0(j,k))  \Big\} \\
&\geq \min\limits_{j\geq 1}\left\{2i-1+\left\lfloor  \frac{5j-5}{3} \right\rfloor +\left\lfloor \frac{5k-j}{3}\right\rfloor, 2i-1+\left\lfloor \frac{5j-3}{3}\right\rfloor +\left\lfloor \frac{5k-j}{3}\right\rfloor \right\} \\
&= 2i-1+\left\lfloor \frac{5k-1}{3}\right\rfloor \\
&=2i+\left\lfloor \frac{5k-4}{3}\right\rfloor.
\end{align*}
Similarly, by \eqref{b-even} and Lemma \ref{lem-mxyz-ord} we have
\begin{align*}
\pi(b(2i,k)) &\geq \min\limits_{j\geq 1} \pi(b(2i-1,j))+\pi(x_1(j,k)) \\
&\geq \min\limits_{j\geq 1} 2i-1+\left\lfloor \frac{5j-3}{3}\right\rfloor +\left\lfloor \frac{5k-j+2}{3}\right\rfloor  \\
&= 2i-1+\left\lfloor \frac{5k+1}{3}\right\rfloor \\
&=2i+1+\left\lfloor \frac{5k-5}{3}\right\rfloor.
\end{align*}
Hence \eqref{a-even-order} and \eqref{b-even-order} holds for $i$.

Next, by \eqref{a-odd} and Lemma \ref{lem-mxyz-ord} we have
\begin{align*}
&\pi(a(2i+1,k)) \\
\geq &~ \min\limits_{j\geq 1} \left\{\pi(a(2i,j))+\pi(y_0(j,k)), \pi(b(2i,j))+\pi(z_0(j,k))  \right\} \\
\geq  &~ \min\limits_{j\geq 1}\left\{2i+\left\lfloor  \frac{5j-4}{3} \right\rfloor +\left\lfloor \frac{5k-j-1}{3}\right\rfloor, 2i+1+\left\lfloor \frac{5j-5}{3}\right\rfloor +\left\lfloor \frac{5k-j-2}{3}\right\rfloor \right\} \\
= &~ 2i+1+\left\lfloor \frac{5k-5}{3}\right\rfloor.
\end{align*}
Similarly, by \eqref{b-odd} and Lemma \ref{lem-mxyz-ord} we have
\begin{align*}
&\pi(b(2i+1,k)) \\
\geq &~ \min\limits_{j\geq 1} \left\{\pi(a(2i,j))+\pi(y_1(j,k)), \pi(b(2i,j))+\pi(z_1(j,k))  \right\} \\
\geq &~ \min\limits_{j\geq 1}\left\{2i+\left\lfloor  \frac{5j-4}{3} \right\rfloor +\left\lfloor \frac{5k-j+1}{3}\right\rfloor, 2i+1+\left\lfloor \frac{5j-5}{3}\right\rfloor +\left\lfloor \frac{5k-j+1}{3}\right\rfloor \right\} \\
= &~ 2i+1+\left\lfloor \frac{5k-3}{3}\right\rfloor.
\end{align*}
Therefore, \eqref{a-odd-order} and \eqref{b-odd-order} hold for $i+1$. By induction we complete the proof.
\end{proof}

Now we are able to prove the main theorems.
\begin{proof}[Proof of Theorem \ref{thm-c(n)}]
Congruences \eqref{c(n)-cong-1} and \eqref{c(n)-cong-2} follow from Theorem \ref{thm-c(n)-gen} and Lemma \ref{lem-ab-order}.
\end{proof}

\begin{proof}[Proof of Theorem \ref{thm-C5}]
From \cite[Corollary 2.7]{Garvan-Shaffer} we find
\begin{align}\label{SC5-zq}
S_{C5}(z,q)=\frac{(-q;q)_\infty (zq,z^{-1}q,q^2;q^2)_{\infty}-(q;q)_\infty}{(-q,z,z^{-1};q)_\infty},
\end{align}
where $S_{C5}(z,q)$ was defined in \cite[Eq.\ (2.6)]{Garvan-Shaffer} as
\begin{align}
S_{C5}(z,q)=\sum_{n=1}^\infty \frac{q^{(n^2+n)/2}(q^{2n+1};q^2)_\infty (q^{n+1};q)_\infty}{(zq^n,z^{-1}q^n;q)_\infty}.
\end{align}
Let
\begin{align}
S_{C5}(q):=\sum_{n=1}^\infty \spt_{C5}(n)q^n.
\end{align}
It follows from \eqref{spt-C5-gen} that $S_{C5}(1,q)=S_{C5}(q)$. We let
\begin{align}
F(z)=(-q;q)_\infty (zq,z^{-1}q,q^2;q^2)_\infty.
\end{align}
Taking logarithmic differentiation on both sides, we obtain
\begin{align}
\frac{F'(z)}{F(z)}=\sum_{n=0}^\infty \left(\frac{-q^{2n+1}}{1-zq^{2n+1}}+\frac{z^{-2}q^{2n+1}}{1-z^{-1}q^{2n+1}} \right). \label{F-diff}
\end{align}
We have $F(1)=(q;q)_\infty$ and $F'(1)=0$. Now we differentiate both sides of \eqref{F-diff} with respect to $z$. We get
\begin{align}
\frac{F''(z)F(z)-(F'(z))^2}{(F(z))^2}=\sum_{n=0}^\infty \frac{q^{2n+1}(-2z^{-3}+z^{-4}q^{2n+1})}{(1-z^{-1}q^{2n+1})^2}-\sum_{n=0}^\infty \frac{q^{4n+2}}{(1-zq^{2n+1})^2}. \label{F-2nd-diff}
\end{align}
Setting $z=1$ in \eqref{F-2nd-diff}, we obtain
\begin{align}
F''(1)=-2(q;q)_\infty \sum_{n=0}^\infty \frac{q^{2n+1}}{(1-q^{2n+1})^2}.
\end{align}
Therefore, by \eqref{SC5-zq} and L'Hospital's rule, we have
\begin{align}
S_{C5}(q)&=-   \lim_{z\rightarrow 1} \frac{F(z)-F(1)}{(z-1)^2}\cdot \frac{1}{(-q;q)_\infty (q;q)_\infty^2} \nonumber \\
&=-\frac{F''(1)}{2(-q,q,q;q)_\infty} \nonumber \\
&=\frac{1}{(q^2;q^2)_\infty}\sum_{n=0}^\infty \frac{q^{2n+1}}{(1-q^{2n+1})^2}. \label{SC5-rep-1}
\end{align}
Since
\begin{align}
\frac{x}{(1-x)^2}=\sum_{n=1}^\infty nx^n,
\end{align}
we have
\begin{align}
&\sum_{n=0}^\infty \frac{q^{2n+1}}{(1-q^{2n+1})^2} \nonumber\\
=&~ \sum_{n=0}^\infty \sum_{m=1}^\infty  mq^{(2n+1)m} \nonumber  \\
=&~ \sum_{m=1}^\infty \sum_{n=0}^\infty mq^{(2n+1)m} \nonumber \\
=&~ \sum_{m=1}^\infty \frac{mq^m}{1-q^{2m}}  \nonumber \\
=&~ \sum_{m=1}^\infty \frac{mq^m}{1-q^m}-\sum_{m=1}^\infty \frac{mq^{2m}}{1-q^{2m}} \nonumber \\
=&~ \frac{1}{24}\left(E_2(2\tau)-E_2(\tau) \right). \label{SC5-add-1}
\end{align}
Thus \eqref{SC5-rep-1} can be rewritten as
\begin{align}
24S_{C5}(q)=\frac{2E_2(2\tau)-E_2(\tau)}{(q^2;q^2)_\infty} -\frac{E_2(2\tau)}{(q^2;q^2)_\infty}. \label{SC5-new-rep}
\end{align}
Taking logarithmic differentiation to both sides of \eqref{p(n)-gen}, and then multiplying by $q$, we obtain
\begin{align}
\sum_{n=1}^\infty np(n)q^n=\frac{1}{(q;q)_\infty}\sum_{n=1}^\infty \frac{nq^n}{1-q^n}=\frac{1-E_2(\tau)}{24(q;q)_\infty}. \label{p(n)-diff}
\end{align}
Substituting \eqref{p(n)-diff} with $q$ replaced by $q^2$ into \eqref{SC5-new-rep}, we conclude that
\begin{align}\label{C5-relation-proved}
24\spt_{C5}(n)=c(n)+(12n-1)p(n/2).
\end{align}
Theorem \ref{thm-C5} then follows from Theorem \ref{thm-c(n)}.
\end{proof}

\begin{proof}[Proof of Theorem \ref{thm-conj}]
Although \eqref{relation} is already known in \cite{Garvan-Shaffer}, for the sake of completeness, we shall provide a new proof here.

From \eqref{SC5-add-1} and \eqref{SC5-rep-1} we can write
\begin{align}
\sum_{n=1}^\infty \spt_{C5}(n)q^n=\frac{1}{(q^2;q^2)_\infty}\sum_{n=1}^\infty \left(\frac{nq^n}{1-q^n}-\frac{nq^{2n}}{1-q^{2n}}  \right). \label{SC5-final}
\end{align}
Combining \eqref{spt-gen} with $q$ replaced by $q^2$, \eqref{spt-omega-new-gen} and \eqref{SC5-final}, we obtain the relation \eqref{relation}.

Note that $\spt_{\omega}(2n+1)=\spt_{C5}(2n+1)$, Conjecture \ref{Wang-conj} follows from \eqref{C5-cong-power-1} and \eqref{C5-cong-power-2}.

Similarly, since $\spt_{\omega}(2n)=\spt(n)+\spt_{C5}(2n)$, congruence \eqref{spt-omega-even} follows from \eqref{spt-5-power} and Theorem \ref{thm-C5}.
\end{proof}
\begin{rem}
We may also prove Conjecture \ref{Wang-conj} without refereing to the congruences of $\spt_{C5}(n)$. Let $\sigma(n)$ denote the sum of positive divisors of $n$. Note that
\begin{align}
\sigma(2n)=3\sigma(n)-2\sigma(n/2).
\end{align}
We have
\begin{align*}
E_2(\tau)&=1-24\sum_{n=1}^\infty \sigma(n)q^n\nonumber \\
&=\left(1-24\sum_{n=1}^\infty\sigma(2n)q^{2n}\right)-24\sum_{n=0}^\infty \sigma(2n+1)q^{2n+1}\nonumber \\
&=3E_2(2\tau)-2E_{2}(4\tau)-24\sum_{n=0}^\infty \sigma(2n+1)q^{2n+1}.
\end{align*}
From this we deduce the following 2-dissection formulas:
\begin{align}
&\sum_{n=0}^\infty c(2n)q^n=\frac{2E_2(2\tau)-E_2(\tau)}{(q;q)_\infty}, \label{c(n)-even}\\
&\sum_{n=0}^\infty c(2n+1)q^n=\frac{24}{(q;q)_\infty}\sum_{n=0}^\infty \sigma(2n+1)q^n. \label{c(n)-odd}
\end{align}
It follows from \eqref{c(n)-odd} and \cite[Eq.\ (4.2)]{Wang-JNT} that $c(2n+1)=24 \spt_{\omega}(2n+1)$. Hence Conjecture \ref{Wang-conj} follows from Theorem  \ref{thm-c(n)}.
\end{rem}

\subsection*{Acknowledgements}
The first author was partially supported by the National Natural Science Foundation of China (11801424), the Fundamental Research Funds for the Central Universities (Project No. 2042018kf0027) and a start-up research grant of the Wuhan University. The second author was partially supported by Grant 102-2115-M-009-001-MY4 of the Ministry of Science and Technology, Taiwan (R.O.C.).

\section{Appendix}
In this section, we present some auxiliary formulas, which can be proved by comparing the orders of vanishing at cusps and Fourier coefficients of the  involved functions. The readers may consult \cite{Gordon-Hughes} and \cite{Xiong} for idea of deductions of these formulas.

Group \uppercase\expandafter{\romannumeral1}:
\begin{align*}
&U_5(F)=F(1+4\cdot 5 t), \\
&U_5(Ft^{-1})=F(-3-3\cdot 5^2 t-5^3 t^2), \\
&U_5(Ft^{-2})=F(1+5^2t+0\cdot t^2 -5^5t^3), \\
&U_5(Ft^{-3})=F(9\cdot 5+9\cdot 5^3 t+0\cdot t^2+0\cdot t^3-5^7t^4), \\
&U_5(Ft^{-4})=F(-51\cdot 5-51\cdot 5^3t +0\cdot t^2+0\cdot t^3+0\cdot t^4-5^9 t^5).
\end{align*}

Group \uppercase\expandafter{\romannumeral2}:
\begin{align*}
&U_5(F\rho)=F(4\cdot 5 t +\rho), \\
&U_5(F\rho t^{-1})=F\Big((1-5^3t^2)+\rho (5^2t)  \Big), \\
&U_5(F\rho t^{-2})=F\Big((-7-7\cdot 5^2 t-5^4 t^2-5^5 t^3)+ \rho (5^4t^2) \Big), \\
&U_5(F\rho t^{-3})=F\Big((5^2+5^4t+0\cdot t^2-5^6 t^3-5^7t^4)+\rho (5^6t^3)   \Big), \\
&U_5(F\rho t^{-4})=F\Big((-3\cdot 5-3\cdot 5^3t+0\cdot t^2+0\cdot t^3-5^8t^4-5^9t^5) +\rho (5^8t^4)  \Big).
\end{align*}

Group \uppercase\expandafter{\romannumeral3}:
\begin{align*}
&U_5(FZ)=F(-5t), \\
&U_5(FZt^{-1})=F\Big((1+7\cdot 5t+5^3t^2)+\rho(-1-5^2t)  \Big), \\
&U_5(FZt^{-2})=F\Big((-4-5^3t-4\cdot 5^3t^2)+\rho (5+4\cdot 5^2 t-5^4 t^2) \Big), \\
&U_5(FZt^{-3})=F\Big((3+8\cdot 5^2t+6\cdot 5^4 t^2+7\cdot 5^5t^3+5^7t^4)+\rho(-2\cdot 5-5^3t+5^5t^2)  \Big), \\
&U_5(FZt^{-4})=F\Big((12\cdot 5+22\cdot 5^2 t-56\cdot 5^4 t^2-22\cdot 5^6 t^3-4\cdot 5^8 t^4-5^9t^5)\\
&\quad \quad \quad +\rho(-8\cdot 5-4\cdot 5^3+6\cdot 5^5 t^2+12\cdot 5^6t^3+2\cdot 5^8t^4) \Big).
\end{align*}

Group \uppercase\expandafter{\romannumeral4}:
\begin{align*}
&U_5(FZ\rho)=F\Big((-2\cdot 5 t)+\rho(5^2t) \Big), \\
&U_5(FZ\rho t^{-1})=F\Big((3\cdot 5t+5^3 t^2)+\rho (-5^2t)  \Big), \\
&U_5(FZ\rho t^{-2})=F\Big((1+2\cdot 5t-5^3t^2)+\rho (5^2t)  \Big), \\
&U_5(FZ\rho t^{-3})=F\Big((-8-5^3t+4\cdot 5^4t^2+8\cdot 5^5t^3+5^7t^4)+\rho(5+0\cdot t-5^5 t^2-5^6t^3) \Big), \\
&U_5(FZ\rho t^{-4})=F\Big((31+9\cdot 5^2 t-36\cdot 5^4t^2-86\cdot 5^5 t^3-3\cdot 5^8 t^4-5^9t^5)\\
&\quad \quad \quad +\rho(-7\cdot 5 +5^3t+2\cdot 5^6 t^2+12\cdot 5^6 t^3+5^8 t^4) \Big).
\end{align*}

\end{document}